\documentclass[11pt]{amsart}
\usepackage{geometry}                
\usepackage{graphicx,stmaryrd}
\usepackage{amssymb,amsmath}
\usepackage{epstopdf}
\usepackage{tikz}
\usepackage{lscape}
\usepackage{color}
\usepackage{todonotes}
\usepackage{mathrsfs}
\usepackage{longtable}

\usepackage{hyperref}

\usepackage[hyperpageref]{backref} 

\theoremstyle{plain}
\newtheorem{theorem}{Theorem}[section]
\newtheorem{conjecture}[theorem]{Conjecture}

\newtheorem{corollary}[theorem]{Corollary}
\newtheorem{proposition}[theorem]{Proposition}
\newtheorem{lemma}[theorem]{Lemma}

\theoremstyle{definition}
\newtheorem{definition}[theorem]{Definition}
\newtheorem{example}[theorem]{Example}

\newtheorem{prob}[theorem]{Problem}

\newcommand{\I}{\mathcal{I}}

\newcommand{\A}{\mathcal{A}}

\newcommand{\NN}{\mathbb{N}}
\newcommand{\ZZ}{\mathbb{Z}}

\newcommand{\C}{\mathcal{C}}

\newcommand{\F}{\mathbb{F}}

\newcommand{\Func}{\mathrm{Func}}

\newcommand\restr[2]{{
  \left.\kern-\nulldelimiterspace 
  #1 
  \vphantom{|} 
  \right|_{#2} 
  }}

\newcommand{\rk}{\mathrm{rk}}

\newcommand{\bs}{\backslash}
\newcommand{\Fl}{\mathcal{F}l}

\newcommand{\gm}{\ensuremath{\mu}}

\title{Chain characteristic polynomials of matroids}
\author{Gary Lazzaro}
\thanks{}
\address{US Naval Academy \footnote[1]{The views expressed in this article are those of the authors and do not reflect the official policy or position of the U.S. Naval Academy, Department of the Navy, the Department of Defense, or the U.S. Government.}\label{foot}\\
  572-C Holloway Rd\\
  Annapolis MD, 21402 USA}
\email[]{lazzaro@usna.edu}

\author{Max Wakefield}
\thanks{}
\address{US Naval Academy\footnotemark[1]\\
  572-C Holloway Rd\\
  Annapolis MD, 21402 USA}
\email[]{wakefiel@usna.edu}

\author{Jason Weiss}
\thanks{}
\address{US Naval Academy\footnotemark[1]\\
  572-C Holloway Rd\\
  Annapolis MD, 21402 USA}
\email[]{m236750@usna.edu}

\begin{document}
\maketitle

\begin{abstract} In this note we introduce a family of polynomials on a matroid derived from chain Tutte polynomials which generalize the classic and ubiquitous characteristic polynomial. We show that the coefficients of these polynomials alternate and present a recursion for any matroid. The classic characteristic polynomials were motivated by the chromatic polynomial on graphs. We define a generalized proper vertex coloring on a graph, which we call coupled multicoloring. We also define a generalized nowhere zero flow, which we call coupled multicommodity flow. Then we show that these chain characteristic polynomials enumerate the number of coupled multicolorings and coupled multicommodity flows on a graph. We conclude by listing multiple problems on enumerative properties of chain characteristic polynomials.\end{abstract}

\section{Introduction}

For $k\in \NN=\{0,1,2,3,\dots \}$ we denote $\{1,2,3,\dots, k\}$ by $[k]$. For any set $S$ we denote the power set of $S$ (the set of all subsets of $S$) by $2^S$ and we use the notation ${S\choose k}=\{X\subseteq S\mid |X|=k\}$. Also, for any subsets, $A,B\subseteq S$, we denote the set of functions from $A$ to $B$ by $\Func(A,B)$ and the set difference as $A-B=\{a\in A| a\notin B\}$. A \emph{matroid} is a pair $M=(\A,\rk)$ where $\A$ is a finite set and $\rk: 2^\A \to \NN$ such that $\rk(\emptyset )=0$, for all $a\in \A$, $\rk(a)\in \{0,1\}$, for all $A\subseteq B\in 2^\A$, $\rk (A)\leq \rk(B)$, and for all $A,B\in 2^\A$, $\rk(A\cap B)+\rk(A\cup B)\leq \rk(A)+\rk(B)$. The \emph{rank} of a matroid $M=(\A,\rk )$ is $\rk(\A)$ and is denoted by $\rk(M)$. For $M=(\A , \rk)$ a matroid the dual matroid is $M^*=(\A,\rk^*)$ where $\rk^*(S)=|S|-(\rk(M)-\rk (E-S))$. Matroids are often considered as generalizations of graphs and abstractions of linear independence of sets of vectors. Whitney defined matroids in \cite{Whitney-35} around the same time he was studying graph colorings. We use \cite{Oxley} and \cite{W76} as our main references for matroid theory.

The classic characteristic polynomial of a matroid $M=(\A, \rk)$ is \[ \chi_M(t)=\sum\limits_{A\subseteq \A }(-1)^{|A|}t^{\rk (M)-\rk(A)}.\] This polynomial is central in matroid and graph theory as it encodes graph coloring and flow information \cite{Birk12} (which we will discuss below), Betti numbers of the complement of a complex realized arrangement associated to the matroid \cite{OS-80}, the chambers of the complement of a real arrangement associated to the matroid \cite{Z75}, and various other geometric enumeration problems like counting points in complements of arrangements over field fields \cite{Ath-96} (for more in depth coverage of the characteristic polynomial see \cite{Ardila-15,Zas87}). 

More recently, there has been significant advances in understanding the coefficients of the characteristic polynomial. Adiprasito, Huh, and Katz proved in \cite{AHK17} a conjecture of Rota, Heron, and Welsh that the coefficients of the characteristic polynomials coefficients are log-concave. Also, recently, the characteristic polynomials appear as a crucial part of the definitions of certain new invariants. In \cite{EPW-16} Elias, Proudfoot, and the second author defined the so called Kazhdan-Lusztig polynomials of a matroid with a recursion involving the characteristic polynomials. In particular, the characteristic polynomial is a $P$-kernel in the sense of Stanley \cite{Stan94} (also see \cite{bren03} and \cite{Proud-18}). Another recent example of the prominence of the characteristic polynomial comes from the matroid motivic zeta function defined by Jensen, Kutler, and Usatine in \cite{JKU-21}. The coefficients of the characteristic polynomial also appear in the volume polynomial of a matroid defined by Eur in \cite{E20}. 

In the literature, there are various generalizations of the characteristic polynomial. Abe, Terao, and the second author in \cite{ATW-07} defined a characteristic polynomial for multiarrangements, but this polynomial coincides with the classic characteristic polynomial in the case of a simple matroid, and it depends on a realizable matroid. In \cite{JW-24} Johnson and the second author defined a generalized characteristic polynomial using a generalized M\"obius function. The generalized characteristic polynomials we define below are different.

Given a matroid $M$ with ground set $\A$ defined by a rank function $\rk$, so $M=(\A ,\rk)$, we call \[ \C_\A^k=\left\{ (A_1,A_2,\dots , A_k)\in \left( 2^\A\right)^k| \ A_1\subseteq A_2\subseteq \cdots \subseteq A_k \subseteq \A\right\}\] the chains in $\A$.
Using these chains, we define the main object of study for this note.

\begin{definition}
The $k^{th}$-chain characteristic polynomial of $M=(\A, \rk)$ is \[ \chi_M^k(t_1,\dots ,t_k)=\sum\limits_{(A_i)\in \C^k_\A} \prod\limits_{i=1}^k (-1)^{|A_i|}t_i^{\rk(M)-\rk(A_i)} .\]
\end{definition}

This is a direct generalization of the classic characteristic polynomials since $\chi_M(t)=\chi_M^1(t)$. The M\"obius polynomial (maybe first defined or named such by Zaslavsky in \cite{Z75} for arrangements) is defined by  \begin{equation}\label{Mobdef}\mathcal{M}_M(s,t)=\sum\limits_{X\leq Y\in L(M)}\gm (X,Y)s^{\rk(X)}t^{\rk (M)-\rk(Y)}\end{equation} where $L(M)$ is the set of flats (subsets of $\A$ that are maximal with respect to rank) and $\mu$ is the M\"obius function defined by $\mu (X ,X)=1$ and \[ \sum\limits_{X\leq Y\leq Z}\gm (X,Y)=0 \] for all $X,Y\in L(M)$. The M\"obius polynomial appeared in the study of certain algebras on directed graphs in \cite{GRSW-05} by Gelfand, Retakh, Serconek, and Wilson and further studied in \cite{RW-09} by Retakh and Wilson. Then Jurrius in \cite{Jur-12} and Johnsen and Verdure in \cite{JV-21} studied M\"obius polynomials for applications in coding theory. Then the second author showed essentially the following using the chain Tutte polynomials (which we review below).

\begin{theorem}[Theorem 1.6, \cite{Wak-23}]\label{Mobius}
For any matroid $M$ the M\"obius polynomial is \[\mathcal{M}_M(s,t)=s^{\rk(M)}\chi^2_M(s^{-1},t).\]
\end{theorem}

In \cite{Bry72} Brylawski presented the deletion-recursion formula of the classic characteristic polynomial and showed that the characteristic polynomial is an evaluation of the Tutte polynomial. The evaluation of the Tutte polynomial is essentially how we come to our definition of the chain characteristic polynomial. 

The Tutte polynomial is one of the most studied invariants on graphs and matroids. The Tutte polynomials coefficients were first written down by Whitney in \cite{Whitney-32}. It was formally defined by Tutte in \cite{Tutte-poly} and generalized to matroids by Crapo in \cite{Crapo-69}. In this work, in order to show the origin of our main objects of study  we need an extension of the classic Tutte polynomial called the chain Tutte polynomial, defined by the second author in \cite{Wak-23}.

\begin{definition}\label{Def-Tutte}

The $k^{th}$ \emph{chain Whitney rank generating polynomial} of $M$ is $$W^k_M((a_i)_1^k;(b_i)_1^k)=\sum\limits_{(S_i)_1^k \in \C^k_\A } \prod\limits_{i=1}^ka_i^{\rk (M)-\rk (S_i)} b_i^{|S_i|-\rk(S_i)}$$  and the $k^{th}$ \emph{chain Tutte polynomial} of $M$ is 
$$T^k_M((x_i)_1^k;(y_i)_1^k)=W^k_M((x_i-1)_1^k;(y_i-1)_1^k).$$ 
\end{definition}

Generalizing Theorem \ref{Mobius} to all chain characteristic polynomials, we get the following by evaluating and reducing $T^k_M$.

\begin{theorem}\label{tutteval}If $M$ is a matroid then $$\chi^k_M(t_1,\dots, t_k)=(-1)^{k\rk (M)}W^k_M((-t_i)_1^k;(-1)_1^k)=(-1)^{k\rk (M)}T^k_M((1-t_i)_1^k;(0)_1^k).$$
\end{theorem}

Ardila and Sanchez in \cite{AS-20} showed that many important matroid invariants are matroid valuations (see \cite{AS-20} or the Appendix of \cite{EHL-23} for details on matroid valuations). Using Theorem 1.5 in \cite{Wak-23} and Theorem \ref{tutteval} we get the following.

\begin{corollary}

For every positive $k$, the $k^{th}$ chain characteristic polynomial is a matroid valuation.

\end{corollary}

Because of Theorem \ref{tutteval} and properties of the chain Tutte polynomials which we outline in section \ref{Tutte-sec} we get the following product formula on chain characteristic polynomials.

\begin{corollary}\label{prodform}

If $M_1$ and $M_2$ are matroids then for all $k>0$ $$\chi^k_{M_1\oplus M_2}(t_1,\dots ,t_k)= \chi^k_{M_1}(t_1,\dots ,t_k)\chi^k_{M_2}(t_1,\dots ,t_k).$$
    
\end{corollary}

Now we examine the recursion given for the chain Tutte polynomials in Theorem \ref{recursionTutte} applied to chain characteristic polynomials. In order to do this, we substitute the evaluation from Theorem \ref{tutteval} into the formula in Theorem \ref{recursionTutte}.

\begin{proposition}

If $M$ is a matroid with ground set $\A$ and $a\in \A$ is not a loop nor a coloop and $k>1$ then

\begin{align*}
\chi^k_{M}(t_1,\dots ,t_k)=&\sum\limits_{j=0}^k sT_{M,a}^{k,j}((1-t_i)_1^k;(0)_1^k) \\
=&\chi^k_{M\bs a}(t_1,\dots ,t_k)-\chi^k_{M/ a}(t_1,\dots ,t_k) \\
 &-\sum\limits_{j=1}^{k-1}\sum\limits_{(S_i)\in \C_{\A -a}^k}(-1)^{\sum\limits_{i=1}^k |S_i|}\left[ \prod\limits_{i=1}^j t_i^{\rk(M\bs a)-\rk_{M\bs a} (S_i)}\right] \left[ \prod\limits_{i=j+1}^k t_i^{\rk(M/ a)-\rk_{M/ a} (S_i)}\right]  . \\ 
\end{align*}
    
\end{proposition}

In \cite{JV-21} Johnsen and Verdure proved that the coefficients of the M\"obius polynomial alternate. We extend this result in a few ways. First in section \ref{Tutte-sec} we define a generalized M\"obius function. Then we compute the sign of that M\"obius function in Theorem \ref{thm-sign} and apply it to the chain characteristic polynomials. 

\begin{corollary}\label{cor-sign}
    The coefficients of $\chi^k_M$ alternate in sign by total degree.
\end{corollary}

Now, we restrict the remainder of our study to graphic matroids. Let $G=(V,E)$ be a (simple) directed graph (in which case then we view $E\subseteq V\times V$) or an undirected graph (in which case then we view $E\subseteq {V\choose 2}$). We use \cite{Boll-98} by  Bollob\'as for our general reference for graph theory. Given a subset of edges $A\subseteq E$ we denote by $c(A)$ the number of connected components of the graph restricted to $A$: $G|_A=(V,A)$. The rank of a subset of edges $A\subseteq E$ is $\rk (A)=|V|-c(A)$ and the rank of the graph is $\rk(G)=\rk(E)$. A \emph{proper vertex coloring} of a graph $G$ is a function $f: V\to \C$ where $\C$ is a finite set (the colors) and for all $\{i,j\}\in E$, $f(i)\neq f(j)$. Then the \emph{chromatic polynomial of $G$} is $\chi_G(t)=$the number of proper vertex colorings of $G$ with $|\C|=t$. In \cite{Birk12} Birkhoff essentially showed the following.

\begin{theorem}[\cite{Birk12}]\label{Birk-thm}
For any undirected simple graph $G$, the chromatic polynomial is equal to the characteristic polynomial, up to a factor, of the matroid associated to the graph: \[ \chi_G(t)=t^{c(E)}\chi^1_{M(G)}(t)\] where $M(G)=(E,\rk)$ and $\rk$ is the graph's rank function.
\end{theorem}

In section \ref{graps-sec} we define a generalized coloring which we call coupled multicoloring (see definition \ref{def-coupled-multi}). This coupled multicoloring is exactly a proper vertex coloring in the case $k=1$. From this we get a coupled multicoloring chromatic function which we denote by $\chi_G^k(t_1,\ldots ,t_k)$ (see definition \ref{multichrom-def}). One of our main results is the following which we prove in section \ref{graps-sec} using an inclusion-exclusion argument.

\begin{theorem}\label{multichrom}
    If $G$ is a simple graph and $k>0$ then \begin{align*}\chi_G^k(t_1,\dots ,t_k) &= (t_1\cdot t_2\cdots t_k)^{c(G)} \chi_{M(G)}^k(t_k, t_{k-1},\ldots ,t_1) \\
    &=(-1)^{\rk (G)}(t_1\cdots t_k)^{c(G)} T^k_G(1-t_k,1-t_{k-1},\dots ,1-t_1;0,\dots ,0).\end{align*}
\end{theorem}

The story for flows on graphs is similar. A \emph{nowhere zero flow} on a directed graph $G$ is a function $f:E\to A$ where $A$ is an Abelian group (viewed additively) such that for all $(i,j)\in E$ we have $f(i,j)\neq 0$ and for all $i\in V$ Kirchhoff's first law holds: $$\sum\limits_{\substack{j\in V\\(i,j)\in E}}f(i,j)=\sum\limits_{\substack{j\in V\\(j,i)\in E}}f(j,i) .$$ Then $Flow^1_G(t)=$the number of non-where zero flows on $G$ with abelian group $A$ such that $|A|=t$ is called the flow polynomial of $G$. Tutte proved the following.

\begin{theorem}[\cite{Tutte-poly}]\label{flow1-poly-thm}

For any undirected simple graph $G$, the flow polynomial is \[ Flow_G^1(t)=(-1)^{|E|-\rk(E)}T_{M(G)}^1(0;1-t) = (-1)^{|E|-\rk(E)} \chi_{M(G)^*}^1(t) .\]
    
\end{theorem}

Note that this shows that $Flow^1_G(t)$ is independent of the edge directions and only depends on the size of the abelian group $|A|=t$. In section \ref{graps-sec} we define a generalized type of nowhere zero flow, defined using $k$ abelian groups $A_1, \ldots , A_k$ which we call coupled multicommodity flow, see definition \ref{def-coupled-multi}. Then we enumerate these flows in a polynomial which we denote by $Flow^k_G(A_1,\ldots ,A_k)$ and call this the coupled multicommodity flow polynomial. Next, we prove a similar result to Theorem \ref{multichrom} for the flow polynomial. The proof is similar but is more delicate due to the reordering of the variables.

\begin{theorem}\label{multiflow-thm}
Let $G=(V,E)$ be an undirected graph and $A_i$ be finite abelian groups with cardinalities $|A_i|=t_i$. Then \begin{align*}\mathrm{Flow}^k_{G}(A_1,\ldots ,A_k) &= (-1)^{k|E|}\sum\limits_{(B_i)\in \C_{E}^k}\prod\limits_{i=1}^k(-1)^{|B_i|}t_i^{|B_i|-\rk(B_i)} \\
& = (-1)^{k(|E|+\rk(G))} T^k_{M(G)}(0,\ldots , 0;1-t_1,\ldots , 1-t_k) .\end{align*}
    
\end{theorem}

Again, we note that these flows are independent of the directions on the graph (as this theorem is stated for undirected graphs) and that these flows only depend on the size of the abelian groups. The ordering of the variables in Theorem \ref{multichrom} versus the ordering in Theorem \ref{multiflow-thm} is directly tied to the definitions \ref{def-coupled-multi} and \ref{couple-flow-def}. It looks as though maybe coupled multicommodity flows are more natural than coupled multicoloring, but in the definitions we could have just reordered the implications and that would switch the variable orderings. We also note that the definitions we present are in a sense dual. This duality can be viewed through the definitions themselves or through Theorems \ref{multichrom} and \ref{multiflow-thm} together with the dual formula in proposition \ref{Tutte-dual}.

We split the remainder of this note into three sections. First, in section \ref{Tutte-sec}, we examine chain characteristic polynomials for matroids in general from the perspective of chain Tutte polynomials. Then, in section \ref{graps-sec}, we study the chain characteristic polynomials entirely in terms of graphs and develop generalized coloring and flow. We finish this note with a section on problems which are natural extensions of results on the classic characteristic polynomials.

\

\begin{flushleft}
{\bf Acknowledgments:} We would like to thank Franklin Kenter for many discussions on the graph theory portions of this work. His comments helped significantly with the proof of Theorem \ref{flow1-poly-thm}. We also want to thank Will Traves for many helpful discussions and comments on various portions of this project.
\end{flushleft}

\section{Chain characteristic polynomials for matroids}\label{Tutte-sec}

In this section, we recall some results on chain Tutte polynomials contained in \cite{Wak-23} to deduce results on chain characteristic polynomials. First, we note that the higher chain Tutte polynomials determine the lower chain Tutte polynomials.

\begin{lemma}[Lemma 3.1 \cite{Wak-23}]\label{getT1}

For any matroid $M$ and any $k\geq 1$ $$T^{k+1}_M(2,2x_1-1,x_2,\dots ,x_k;2,2^{-1}y_1+2^{-1},y_2,\dots ,y_k)=2^{\rk (M)}T^k_M(x_1,\dots ,x_k;y_1,\dots ,y_k).$$

\end{lemma}

Next, we note that the chain Tutte polynomial of the dual matroid is the same as a change of variables of the chain Tutte polynomial of the matroid.

\begin{proposition}[Proposition 3.4 \cite{Wak-23}]\label{Tutte-dual}

If $M$ is a matroid then $$T^k_{M^*}(x_1,\dots ,x_k;y_1,\dots , y_k)=T^k_M(y_k,\dots ,y_1;x_k,\dots ,x_1).$$

\end{proposition}

The next fact is that chain Tutte polynomials are multiplicative over direct sums of matroids. Recall that if $M_1=(\A_1,\rk_1)$ and $M_2=(\A_2,\rk_2)$ are matroids then $M_1\oplus M_2=(\A_1 \sqcup \A_2,\rk_\oplus)$ where $\rk_\oplus (A)=\rk_1(A\bigcap \A_1)+\rk(A\bigcap \A_2)$. 

\begin{proposition}[Proposition 3.5 \cite{Wak-23}]\label{Tprod}

If $M_1=(\A_1,\rk_1)$ and $M_2=(\A_2,\rk_2)$ are matroids then $$T^k_{M_1\oplus M_2}((x_i);(y_i))=T^k_{M_1}((x_i);(y_i))T^k_{M_2}((x_i);(y_i)).$$

\end{proposition}

Again, following \cite{Wak-23} we have a recursion for chain Tutte polynomials. This is given by the \emph{split chain Tutte polynomials} from \cite{Wak-23} which for $0<j<k$ are defined by \begin{align*} sT^{k,j}_{M,a}((x_i),(y_i)) &=\sum\limits_{(S_i)_{1}^{j} \in \C^{j}_{\A -a} } \prod\limits_{i=1}^{j}(x_{i}-1)^{\rk(M\bs a)-\rk_{M\bs a}(S_i)} (y_{i}-1)^{|S_i|-\rk_{M\bs a}(S_i)}\\
&\cdot \sum\limits_{\substack{(S_i)_{j+1}^{k} \in \C^{k-j}_{\A -a}\\ S_j\subseteq S_{j+1}} } \prod\limits_{i=j+1}^{k}(x_{i}-1)^{\rk(M/ a)-\rk_{M/ a}(S_i)} (y_{i}-1)^{|S_i|-\rk_{M/ a}(S_i)}\\
&=\sum\limits_{S_1\subseteq \A-a}g(j,a,S_1)T^j_{M\mid_{S_1}}((x_i)_1^j;(y_i)_1^j)T^{k-j-1}_{M/(S_1\cup a)}((x_i)_{2+j}^k;(y_i)_{2+j}^k)\end{align*} where \begin{equation*}g(j,a,S_1)=\left[ \prod\limits_{i=1}^j (x_i-1)^{\rk (M/S_1)}\right] (x_{j+1}-1)^{\rk(M/a)-\rk_{M/a}(S_1)}\left[ \prod\limits_{i=1}^{k-j} (y_{i+j}-1)^{|S_1|-\rk_{M/a}(S_1)}\right] .\end{equation*} Then for $j=0$ the split chain Tutte polynomial  $sT_{M,a}^{k,0}=T^k_{M/a}$ and for $j=k$ it is $sT_{M,a}^{k,k}=T^k_{M\backslash a}$.

Then the recursion with these split chain Tutte polynomials is given by the following.

\begin{theorem}[Theorem 1.3 \cite{Wak-23}]\label{recursionTutte}

For a matroid $M=(\A,\I)$ and $a\in \A$ that is not a loop and not a coloop \[T^k_M=\sum\limits_{j=0}^k sT^{k,j}_{M,a}.\]

\end{theorem}

Theorem \ref{recursionTutte} gives a recursion, albeit rather complicated, for the chain characteristic polynomials. 

\begin{corollary}

For a matroid $M=(\A,\I)$ and $a\in \A$ that is not a loop and not a coloop

\[\chi^k_M(t_1,\ldots ,t_k) = \chi^k_{M\bs a}(t_1,\ldots ,t_k)+(-1)^{k} \chi^k_{M/a}(t_1,\ldots ,t_k)\]
\[+\sum\limits_{j=1}^{k-1} \sum\limits_{S_1\subseteq \A-a} (-1)^{s(j,a,S_1)}p(j,a,S_1) \chi^{j}_{M\mid_{S_1}}(t_1,\ldots ,t_j) \chi^{k-j-1}_{M/(S_1\cup a)}(t_{2+j},\ldots ,t_k)   \]

where

\[ p(j,a,S_1)= t_{j+1}^{\rk(M/a)-\rk_{M/a}(S_1)}\prod\limits_{i=1}^j t_i^{\rk (M/S_1)}\]

and 

 \begin{align*} s(j,a,S_1)&= \rk(M/a)+(k-j)|S_1|+j(\rk(M/S_1)+\rk(M|_{S_1}))\\
 & +(k-j-1)\rk(M/(S_1\cup a))-(k-j+1)\rk_{M/a}(S_1) .\end{align*} 
\end{corollary}

Now we examine the coefficients of chain characteristic polynomials. We use these coefficients to define a generalized M\"obius function and show that these coefficients alternate. Recall that given a matroid the lattice of flats is $L(M)$ where we denote join of a set $S\subseteq \A$ by $\vee S=F$ if $F$ is the minimal flat that contains all the elements of $S$.

\begin{definition}
    Let $M=(\A, \rk)$ be a matroid, $L(M)$ its lattice of flats, and \[\Fl^k_M=\left\{ (X_1,\dots ,X_k)\in L(M)^k|\ X_i\subseteq X_{i+1}\right\} .\]
We define the $k^{th}$-chain M\"obius function by \[ \mu^k : \Fl^k_M \to \ZZ\] by \[\mu^k(X_1,\dots , X_k)=\sum\limits_{\substack{(A_i)\in \C_{\A}^k\\ \vee A_i=X_i}}\prod\limits_{i=1}^k(-1)^{|A_i|} .\]
\end{definition}

If we denote by $\mu$ the classic M\"obious function, then $\mu^1(X)=\mu (\hat{0},X)$ and $\mu^2=\mu$ as a consequence of \cite[Lemma 2.2]{Wak-23}.

By reordering the coefficients of the chain Tutte polynomial, we get the following M\"obius function decomposition of the chain characteristic polynomials.

\begin{proposition}
   For any matroid $M$ \[\chi_M^k(t_1,\dots ,t_k)= \sum\limits_{(X_i)\in \Fl^k_M}\hspace{-.4cm}\mu^k(X_1,\dots ,X_k)\prod\limits_{i=1}^kt_i^{\rk(M)-\rk(X_i)} . \]
\end{proposition}

Next we look at the sign of the chain M\"obius function. 

\begin{theorem}\label{thm-sign}

For any $(X_1,\dots ,X_k)\in \Fl_M^k$ where $M$ is a simple matroid \[\mathrm{sgn}(\mu^k(X_1,\dots ,X_k))=\prod\limits_{i=1}^k(-1)^{\rk(X_i)} .\]
    
\end{theorem}

\begin{proof}
We prove this by induction on $k$. The base case for $k=1$ (and $k=2$) is due to Rota in \cite[Theorem 4, pg. 357]{rota64}. We assume the result is true for $k-1$ and any matroid. So, we compute the $k^{th}$ chain M\"obius function on any chain of flats 
\begin{align*} 
\mu^k (X_1,\dots ,X_k) & =\sum\limits_{\substack{(A_i)\in \C_{\A}^k\\ \vee A_i=X_i}}\prod\limits_{i=1}^k(-1)^{|A_i|}\\
&=\sum\limits_{\substack{A_k\subseteq \A\\ \vee A_k=X_k}}(-1)^{|A_k|} \sum\limits_{\substack{(A_i)\in \C_{A_k}^{k-1}\\ \vee A_i=X_i}}\prod\limits_{i=1}^{k-1}(-1)^{|A_i|}\\
&=\sum\limits_{\substack{A_k\subseteq \A\\ \vee A_k=X_k}}(-1)^{|A_k|} \mu^k_{M|_{A_k}}(X_1,\dots ,X_{k-1}).\\ 
\end{align*}

Applying the sign function to the above and using induction finishes the proof. \end{proof}

\begin{example}\label{Boolean-chi}

Let $U_{n,n}$ be the uniform matroid with $n$ elements of rank $n$. These matroids have many names, for example Boolean matroids and free matroids. Using \cite[Example 3.6]{Wak-23}, Theorem \ref{tutteval} and Corollary \ref{prodform} we have that \[\chi^k_{U_{n,n}}(t_1,\ldots ,t_k)= \left( (-1)^{k} \left(1+ \sum\limits_{i=1}^k(-1)^i\prod\limits_{j=1}^it_j \right)\right)^n.\]

\end{example}

\section{Chain characteristic polynomials on graphs}\label{graps-sec} 

The original motivation for the classic characteristic polynomial comes from Whitney in \cite{Whitney-32} on the proper vertex coloring of graphs. We aim to preserve this connection by examining generalized notions of proper vertex colorings on a graph. However, we note that there may be other, more canonical, definitions for these generalizations. We first study vertex coloring in subsection \ref{vertex-sec} and then study generalized edge flows in subsection \ref{flow-sec}.

\subsection{Coupled multicoloring of graphs}\label{vertex-sec} Let $G=(V,E)$ be a graph with $|V|=n$ vertices. We note that when considering induced graphs on edges $G|_A$ these have the same number of vertices $n$.

\begin{definition}\label{def-coupled-multi}

A \emph{$k$-multicoloring} of an undirected graph $G$ is a $k$-tuple of functions $(f_1,\dots ,f_k)$, $f_i:V\to \C_i$ where the $\C_i$ are sets. For $k>1$ a \emph{coupled $k$-multicoloring} of an undirected graph $G$ is $k$-multicoloring such that for all edges $\{a,b\}\in E$ if $k=2\ell$ is even then the coloring functions must satisfy all of the following implications \[\begin{array}{l}

f_1(a)=f_1(b) \Rightarrow   f_2(a)=f_2(b)\\

f_1(a)=f_1(b) \text{ and } f_3(a)=f_3(b) \Rightarrow  f_4(a)=f_4(b)\\

\vdots \\

f_1(a)=f_1(b) \text{ and } \cdots \text{ and } f_{k-1}(a)=f_{k-1}(b) \Rightarrow  f_{k}(a)=f_{k}(b)\\

\end{array}
\] and in the case if $k=2\ell +1>1$ is odd then the coloring functions must satisfy all the following implications \[\begin{array}{l}

f_1(a)=f_1(b) \Rightarrow   f_2(a)=f_2(b)\\

f_1(a)=f_1(b) \text{ and } f_3(a)=f_3(b) \Rightarrow  f_4(a)=f_4(b)\\

\vdots \\

f_1(a)=f_1(b) \text{ and } \cdots \text{ and } f_{k-2}(a)=f_{k-2}(b) \Rightarrow  f_{k-1}(a)=f_{k-1}(b)\\

f_1(a)=f_1(b) \text{ and } \cdots \text{ and } f_{k-2}(a)=f_{k-2}(b) \Rightarrow  f_{k}(a)\neq f_{k}(b).\\

\end{array}
\] 

\end{definition}

One way to phrase this coupled $k$-multicoloring is to say that all the odd numbered coloring functions force the evens up to one higher to also be equal and in the odd case we also have the assumption that if all the odds are equal then the last is forced to not be equal. We note that the classical case, $k=1$, can be viewed by looking at the last line of Definition \ref{def-coupled-multi} as just a proper vertex coloring.

\begin{example}

Consider the complete graph $K_3=([3],E)$ on three vertices (where $E=\{\{a,b\} | a\neq b\in [3] \}$) and the pair of functions $f_1:V\to [2]$ and $f_2:V\to [2]$ where $f_1(1)=f_1(2)=1=f_2(3)$ and $f_2(1)=f_2(2)=2=f_1(3)$. The pair of functions $(f_1,f_2)$ is a coupled $2$-multicoloring of $K_3$ and we depict this coloring in Figure \ref{K_3-multi}. However, if we defined functions $g_1:V\to [2]$ and $g_2:V\to [2]$ where $g_1(1)=g_1(2)=1=g_2(1)=g_2(3)$ and $g_2(1)=2=f_1(3)$ then this pair would not be a coupled $2$-multicoloring.

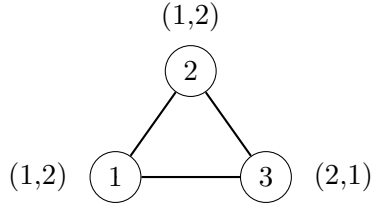
\begin{figure}

\begin{tikzpicture}
    \node[below,left] at (-.5,0) {(1,2)};
\node[above] at (1,1.8) {(1,2)};
\node[right] at (2.5,0) {(2,1)};

    \node[shape=circle,draw=black] (1) at (0,0) {1};
\node[shape=circle,draw=black] (2) at (1,1.41) {2};
\node[shape=circle,draw=black] (3) at (2,0) {3};
\path [thick] (1) edge (2) edge (3) edge (1);
\path [thick] (3) edge (2);
    
\end{tikzpicture}\caption{$K_3$ labeled with a coupled 2-multicoloring}\label{K_3-multi}

\end{figure}

\end{example}

Now with the definition of coupled $k-$multicoloring we can define coupled chromatic function defined for a graph. Which we will then prove that this function is actually a polynomial.

\begin{definition}\label{multichrom-def}
For $k>1$ the \emph{k-coupled chromatic function} on a simple graph $G$ is the function 
\[ \chi_G^k:\mathbb{Z}_{>0}^k\to \mathbb{N}\] defined by $\chi_G^k(t_1,\dots ,t_k)=$the number of coupled $k$-multicolorings of $G$ with color sets sizes $t_1,\dots ,t_k$. If $k=1$ then we set $\chi_G^1=\chi_G$ to be the classic chromatic polynomial.
\end{definition}

Now we prove the main theorem of this subsection, Theorem \ref{multichrom}.

\begin{proof}[Proof of Theorem \ref{multichrom}]
We progress by induction on $k$. The base case $k=1$ is the classic formula found in Theorem \ref{Birk-thm}. Now suppose $k>1$ and set \[U_k=U_k(t_1,\ldots , t_k)=\prod\limits_{i=1}^k \Func (V,[t_i])\] to be the universe of all $k$-multicolorings of $G$ with $(t_1,\ldots,t_k)\in \ZZ^K_{\geq 1}$ many colors. Note that here and below we may leave out the color sets $[t_i]$ for convenience and compactness of notation. For each edge $\{a,b\} \in E$ we first let \[ P_{\{a,b\}}^1(t_1)=\left\{ f_1\in U_1(t_1) | \ f_1(a)= f_1(b) \right\} . \] These are the $1$-multicoloring functions which are strictly not coupled on the edge $\{a,b\}$ (a.k.a. not proper vertex coloring). Then for each $1<m\in \mathbb{Z}$ and each edge $\{a,b\} \in E$ we define the following sets recursively by 
\[ P_{\{a,b\}}^m(t_1,\ldots ,t_m)=\left\{  (f_1,\dots ,f_m)\in U_m\left|  \ f_1\in P_{\{a,b\}}^1 \text{ and } (f_2,\dots , f_m)\in \overline{P_{\{a,b\}}^{m-1}} \right. \right\} \] where $\overline{P_{\{a,b\}}^{m-1}}(t_2,\ldots ,t_m)=U_{m-1}(t_2,\ldots ,t_m)-P_{\{a,b\}}^{m-1}(t_2,\ldots ,t_m)$ is the compliment of the shifted up color sets. The set $\overline{P^k_{\{a,b\}}}$ is exactly the set of coupled $k$-multicolorings of the edge $\{a,b\}$ presented in Definition \ref{def-coupled-multi}. In order to get the total number of coupled $k-$multicolorings of the entire graph $G$ we can express \begin{equation}\label{cardpresent} \chi_G^k(t_1,\dots ,t_k)= \left| \bigcap\limits_{e\in E}\overline{P^k_e} \right| .\end{equation} Also, by construction we have \[ \bigcap\limits_{e\in A} P^k_{e}=\bigcap\limits_{e\in A}P^1_{e} \times \bigcap\limits_{e\in A}\overline{P^{k-1}_{e}} \] where $A\subseteq E$ is a subset of edges. Next, we use inclusion-exclusion applied to (\ref{cardpresent}) to get \begin{align}  \chi_G^k(t_1,\dots ,t_k)&=\sum\limits_{A_k\subseteq E}(-1)^{|A_k|}\left| \bigcap\limits_{e\in A_k} P^k_e \right|  \\
&=\sum\limits_{A_k\subseteq E}(-1)^{|A_k|} \left| \bigcap\limits_{e\in A_k}P^1_{e} \times \bigcap\limits_{e\in A_k}\overline{P^{k-1}_{e}} \right|  \\
& = \sum\limits_{A_k\subseteq E}(-1)^{|A_k|} t_1^{c(A_k)}\chi^{k-1}_{G|_{A_k}}(t_2,\ldots ,t_k) \label{colorsetsdecomp}
\end{align} because once all the edges in the induced graph $G|_{A_k}$ are forced to have the same color in $\bigcap\limits_{e\in A_k}P^1_{e}$ the number of choices for colors is the number of components. Using the fact that $c(A_i)=n-\rk(A_i)=\rk(G)+c(G)-\rk(A_i)$ and applying induction to (\ref{colorsetsdecomp}) we have 

 \begin{align*}  \chi_G^k(t_1,\dots ,t_k)&=\sum\limits_{A_k\subseteq E}(-1)^{|A_k|} t_1^{c(G)+\rk(G)-\rk(A_k)} (t_2\cdots t_k)^{c(G)}\chi_{M(G|_{A_k})}^k(t_k,\ldots ,t_2) \\ 
 &=(t_1\cdots t_k)^{c(G)}\sum\limits_{A_k\subseteq E}(-1)^{|A_k|} t_1^{\rk(G)-\rk(A_k)} \sum\limits_{(A_i)\in \C^{k-1}_{A_k}} \prod\limits_{i=1}^{k-1} (-1)^{|A_i|}t_{k-i+1}^{\rk(G)-\rk(A_i)} \\
 &=(t_1\cdots t_k)^{c(G)}\sum\limits_{(A_i)\in \C_{E}^k}\prod\limits_{i=1}^k (-1)^{|A_i|}t_{k-i+1}^{\rk(G)-\rk(A_i)} \\
 \end{align*} which is the desired result. To finish the proof, we apply Theorem \ref{tutteval} in the matroid interpretation $M(G)$. \end{proof}

Next, we compute the chain characteristic polynomials for a few important examples of graphs.

\begin{example}\label{color-com}
    Let $K_n$ be the complete graph with $n$ vertices. Using the computer algebra system Sage \cite{sage} we find that for the first few non-trivial terms $n=3,4,5$:

\begin{align*}
\chi^2_{K_3}(t_1,t_2)=&t_1^2  t_2^2 - 3 t_1^2 t_2 + 2 t_1^2 + 3 t_1 t_2 - 3 t_1 + 1,    \\
\chi^2_{K_4}(t_1,t_2)=&t_1^3  t_2^3 - 6 t_1^3  t_2^2 + 11 t_1^3 t + 6 t_1^2  t_2^2 - 6 t_1^3 - 18 t_1^2 t_2 + 12 t_1^2 + 7 t_1 t_2 - 7 t_1 + 1, \\
\chi^2_{K_5}(t_1,t_2)=&t_1^4  t_2^4 - 10 t_1^4  t_2^3 + 35 t_1^4  t_2^2 + 10 t_1^3  t_2^3 - 50 t_1^4 t_2 - 60 t_1^3  t_2^2 +\\
&24 t_1^4 + 110 t_1^3 t_2 + 25 t_1^2  t_2^2 - 60 t_1^3 - 75 t_1^2 t_2 + 50 t_1^2 + 15 t_1 t_2 - 15 t_1 + 1.
\end{align*}
\end{example}


\subsection{Coupled multicommodity flows}\label{flow-sec} Now we study an interpretation of the characteristic polynomial of the dual matroid as flow on the graph. 

\begin{definition}\label{couple-flow-def}
Let $G=(V,E)$ be an undirected graph and then fix some orientations on the edges in order to consider $G$ as a directed graph denoted by $G'=(V,E')$ where $E'\subseteq V\times V$ such that the underlying edges in $E'$ are the same as in $E$. For $k>1$ a \emph{coupled $k$-multicommodity flow} on $G'$ is an element $(f_1,\ldots ,f_k)\in \prod_{i=1}^k\Func (E,A_i)$ such that each $f_i$ satisfies Kirchhoff’s first law and where for all $1\leq i \leq k$, $A_i$ is an abelian group (viewed additively where $0$ is the identity element) such that for all $(a,b)\in E'$ in the case $k=2m$ is even the flow functions satisfy the implications

\[\begin{array}{l}

f_1(a,b)=0 \Rightarrow   f_2(a,b)=0\\

f_1(a,b)=0 \text{ and } f_3(a,b)=0 \Rightarrow  f_4(a,b)=0\\

\vdots \\

f_1(a,b)=0 \text{ and } \cdots \text{ and } f_{k-1}(a,b)=0 \Rightarrow  f_{k}(a,b)=0.\\

\end{array}\] and in the case $k=2m+1$ is odd, the flow functions satisfy the implications \[\begin{array}{l}

f_1(a,b)=0 \Rightarrow   f_2(a,b)=0\\

f_1(a,b)=0 \text{ and } f_3(a,b)=0 \Rightarrow  f_4(a,b)=0\\

\vdots \\

f_1(a,b)=0 \text{ and } \cdots \text{ and } f_{k-2}(a,b)=0 \Rightarrow [f_{k-1}(a,b)=0 \text{ and } f_{k}(a,b)\neq  0].\\

\end{array}\]

\end{definition}

If $k=1$ we consider the usual nowhere zero flow as our interpretation of coupled 1-multicommodity flow.  Now in order to enumerate flows we define a function which counts all coupled multicommodity flows.

\begin{definition}\label{flow-poly-def}
    Again, given $G=(V,E)$ an undirected graph with some orientations on the edges fixed $G'=(V,E')$. For $k\geq 1$ the coupled $k$-multicommodity flow function on $G'$ is  $\mathrm{cFlow}^k_{G'}(A_1,\ldots ,A_k)=$the number of all coupled $k$-multicommodity flows on $G'$ to the abelian groups $A_i$.
\end{definition}

Our first result shows that $\mathrm{cFlow}^k_{G'}(A_1,\ldots ,A_k)$ does not depend on the orientation of the edges.

\begin{lemma}

If $G'$ and $G''$ are different orientations of the same underlying graph $G$ then \[ \mathrm{cFlow}^k_{G'}(A_1,\ldots ,A_k)=\mathrm{cFlow}^k_{G''}(A_1,\ldots ,A_k).\]

\end{lemma}

\begin{proof}
Let $G'$ be an orientation of $G$ and suppose that $(a,b)\in E'$. Consider the directed graph $\tilde{G}=(V,\tilde{E})$ whose edges are the same as $G'$ except replace $(a,b)$ with $(b,a)$. We build a bijection from the set of coupled $k$-multicommodity flows on $G'$ to those on $\tilde{G}$. Let $(f_1,\ldots ,f_k)$ be a coupled $k$-multicommodity flow on $G'$. Define $\varphi (f_1,\ldots ,f_k)=(\tilde{f}_1,\ldots , \tilde{f}_k)$ where for all $1\leq i\leq k$ and all $(x,y)\in \tilde{E}-(b,a)$, $\tilde{f}_i(x,y)=f_i(x,y)$ and $\tilde{f}_i(b,a)=-f_i(a,b)$. Since for all $1\leq i\leq k$, $A_i$ are abelian groups, we have that the uniqueness of the additive inverse and the negation of zero is zero. Hence, $\varphi (f_1,\ldots ,f_k)$ is a coupled $k$-multicommodity flow on $\tilde{G}$ and $\varphi$ is a bijection. \end{proof}

This lemma allows us to define coupled multicommodity flows on an undirected graph.

\begin{definition}
Let $G$ be an undirected graph and $G'$ be any orientation of $G$. Then the number of multicommodity flows on $G$ to the abelian groups $A_1,\ldots , A_k$ is \[\mathrm{Flow}^k_{G}(A_1,\ldots ,A_k)=\mathrm{cFlow}^k_{G'}(A_1,\ldots ,A_k).\]
\end{definition}

Now we prove the main formula for coupled multicommodity flows, which as a corollary also implies that the number of flows only depends on the sizes of the abelian groups.

\begin{proof}[Proof of Theorem \ref{multiflow-thm}] Similarly to coupled multicoloring we use induction on $k$. The base case is the classic result: Theorem \ref{flow1-poly-thm}. Then for a fixed $k>1$ we assume the formula holds. For $1\leq i\leq m$ and abelian groups $A_i$ we set \[V_m=\left\{ (f_1,\ldots,f_k)\in \prod\limits_{i=1}^m\Func (E,A_i)\right| \left. \forall i,\forall v\in V,\ \sum\limits_{(a,v)\in E}f_i(a,v)=\sum\limits_{(v,b)\in E}f_i(v,b) \right\}.\] Note that we can adorn these sets with the abelian groups by $V_m=V_m(A_1,A_2,\dots ,A_m)$.

For $e\in E$ let \[Q_e^1=Q_e^1(A_1) =\{ f_1\in V_1| f_1(e)=0 \}\] be the set of flows on $G$ which are zero on the edge $e$. Then we can recursively define flow sets by \[ Q_e^m=Q_e^m(A_1,\ldots ,A_m)=\left\{ (f_1,\ldots ,f_m)\in V_m \left| f_1\in Q^1_e \text{ and } (f_2,\ldots ,f_m)\in \overline{Q_e^{m-1}} \right. \right\} .\]
For $S_1\subseteq E$ \[ \bigcap\limits_{e\in S_1} Q_e^k(A_1,\ldots,A_k)=\bigcap\limits_{e\in S_1} Q_e^1(A_1) \times \bigcap\limits_{e\in S_1} \overline{Q_e^{k-1}(A_2,\ldots ,A_k)} .\] Now applying inclusion-exclusion and this product principle to the number of coupled $k$-multicommodity flows we have

\begin{align}  \mathrm{Flow}^k_{G}(A_1,\ldots ,A_k)&=\sum\limits_{S_1\subseteq E}(-1)^{|S_1|}\left|\bigcap\limits_{e\in S_1} Q_e^k\right|\label{PIE} \\
&=\sum\limits_{S_1\subseteq E}(-1)^{|S_1|}\left|\bigcap\limits_{e\in S_1} Q_e^1 (A_1)\right| \cdot \left|\bigcap\limits_{e\in S_1} \overline{Q_e^{k-1} (A_2,\ldots ,A_k)}\right| . \label{prodsum} \end{align} The first term of (\ref{prodsum}) sets all flows to be zero on $S_1$ but leaves the flows free (up to Kirchhoff's first law) on the complement of $S_1$. Using this and another principle of inclusion-exclusion on the last term (\ref{prodsum}) becomes

\begin{equation}\sum\limits_{S_1\subseteq E}(-1)^{|S_1|} t_1^{|E-S_1|-\rk(E-S_1)} \sum\limits_{S_2\subseteq S_1} (-1)^{|S_2|} \left|\bigcap\limits_{e\in S_2} Q_e^{k-1}(A_2,\ldots ,A_k)\right| . \label{innersum}\end{equation}
This last inner summation is the coupled $(k-1)$-multicommodity flow restricted to $S_1$ (it is the same as (\ref{PIE}) but with $k-1$) but that satisfy Kirchhoff's first law on the entire graph. Hence, by induction (\ref{innersum}) becomes \begin{equation}\sum\limits_{S_1\subseteq E}(-1)^{|S_1|} t_1^{|E-S_1|-\rk(E-S_1)} (-1)^{(k-1)|E|}\sum\limits_{(S_i)\in \C^{k-1}_{S_1}} \prod\limits_{i=2}^k(-1)^{|S_i|}t_i^{|E-S_i|-\rk(E-S_i)} \label{justpolys} \end{equation} where this last inner sum is indexed in reversed order, so $(S_i)\in \C^{k-1}_{S_1}$ here means $S_k\subseteq S_{k-1}\subseteq \cdots \subseteq S_2 \subseteq S_1$. Reindexing via complements by setting $B_i=E-S_i$ the reverse chain $S_k\subseteq S_{k-1}\subseteq \cdots \subseteq S_2 \subseteq S_1$ is $B_1\subseteq \cdots \subseteq B_k$. Also using the complement $(-1)^{|S_1|}=(-1)^{|E|}(-1)^{|E-S_1|}$. Putting this all together in (\ref{justpolys}) we have the coupled $k$-multicommodity flow polynomial is \begin{equation*}
(-1)^{k|E|}\sum\limits_{(B_i)\in \C_{E}^k}\prod\limits_{i=1}^k (-1)^{|B_i|}t_{i}^{|B_i|-\rk(B_i)} \\
 \end{equation*} which is the desired result since $(B_i)$ is not a reversed chain. To finish the proof, we apply Theorem \ref{tutteval} in the matroid interpretation. \end{proof}

We finish this section by contrasting Example \ref{color-com} on coupled multicoloring to that of the same examples with coupled multicommodity flow.

\begin{example}\label{flow-com}
    Again using the computer algebra system Sage \cite{sage} the coupled 2-multicommodity flow polynomials for the first few non-trivial complete graphs are

\begin{align*}
\mathrm{Flow}^2_{K_3}(t_1,t_2)=&t_1t_2 - t_2 + 1,    \\
\mathrm{Flow}^2_{K_4}(t_1,t_2)=&t_1^3t_2^3 - 6t_1^2t_2^3 + 6t_1^2t_2^2 + 11t_1t_2^3 - 18t_1t_2^2 - 6t_2^3 + 7t_1t_2 + 12t_2^2 - 7t_2 + 1, \\
\mathrm{Flow}^2_{K_5}(t_1,t_2)=& t_1^6   t_2^6 - 10   t_1^5   t_2^6 + 10   t_1^5   t_2^5 + 45   t_1^4   t_2^6 - 90   t_1^4   t_2^5 - 115   t_1^3   t_2^6 + 45   t_1^4   t_2^4 +\\
&340   t_1^3   t_2^5 + 175   t_1^2   t_2^6 - 330   t_1^3   t_2^4 - 670   t_1^2   t_2^5 - 147   t_1   t_2^6 + 105   t_1^3   t_2^3 + 930   t_1^2   t_2^4 +\\
&670   t_1   t_2^5 + 51   t_2^6 - 550   t_1^2   t_2^3 - 1155   t_1   t_2^4 - 260   t_2^5 + 115   t_1^2   t_2^2 + 925   t_1   t_2^3 + 510   t_2^4\\
&- 330   t_1   t_2^2 - 480   t_2^3 + 37   t_1   t_2 + 215   t_2^2 - 37   t_2 + 1.
\end{align*} In the case of $K_4$, this calculation illustrates the duality formula in Proposition \ref{Tutte-dual} and the variable swap since $K_4$ is selfdual.
\end{example}

\section{Problems}

Due to the ubiquity of characteristic polynomials in matroid and arrangement theory we list some possible projects for future research. The basic idea is to just generalize classic results on the characteristic polynomial $\chi^1$. Many of these problem overlap in some significant ways. 

\subsection{Chambers and faces of real arrangements} Given a real hyperplane arrangement $\A$, Zaslavsky proved in \cite[Theorem A]{Z75} that the M\"obious polynomial $\mathcal{M}_\A(s,t)$ contained the information of the faces (i.e. the real components of the arrangements partition of Euclidean space) of $\A$ in $\mathbb{R}^r$. In particular, Zaslavsky showed that \[ f_\A(x)=(-1)^r\mathcal{M}(-x,-1)\] where $f_\A(x)=\sum\limits_{i=0}^rf_ix^{r-i}$ is the face polynomial of $\A$ (i.e. $f_i$ is the number of $i$ dimensional faces). Because of Theorem \ref{Mobius} we then have that \[f_\A(x)= x^r\chi_\A^2(-x^{-1},-1).\] From this we can also see that $\chi_\A^2(-1,-1)$ gives the total number of faces of $\A$. In the case of the complete graph we have that $\chi^2_{M(K_n)}(-1,-1)$ is the $n^{th}$ ordered Bell number or Fubini number. This leads us to the following problem.

\begin{prob}
    Is there an interpretation of faces of a real arrangement (possibly relative) which is enumerated by the higher characteristic polynomials $\chi^k_\A$? In particular, what does $\chi^k_\A(-1,\ldots ,-1)$ count for real arrangements?
\end{prob}

Note that Example \ref{Boolean-chi} gives that \[ \chi_{U_{n,n}}^k(-1,\ldots ,-1)=(1+k)^n.\]

\subsection{Orlik-Solomon algebras} If $M$ is a realizable matroid over the complex numbers by the hyperplane arrangement $\A$ the classic result by Orlik and Solomon in \cite{OS-80} is that  \[(-1)^{\rk(M(\A))}\chi_M^1(-t^{-1})=\sum\limits_{i=0}^{\rk (M)}\dim (H^i(U(\A)))t^i\] where $U(\A)$ is the complex complement of $\A$ (this can be stated more generally in terms of the Orlik-Solomon algebra as can be seen in \cite[Theorem 3.68]{OT}). This result together with Corollary \ref{cor-sign} leads us to propose the following problem.

\begin{prob}
    Is there a multigraded algebra associated to a matroid $M$, possibly generalizing the Orlik-Solomon algebra, who's multigraded Hilbert series is equal to the evaluation of the chain Tutte polynomials $(-1)^{k \rk (M)} \chi^k_M(-t^{-1}_1,\dots , -t^{-1}_k)$?
\end{prob}

\subsection{Arrangements over finite fields} If $M$ is realizable over a finite field $\F_q$ by an arrangement $\A\subseteq \F_q^n$ then Athanasiadis showed in \cite{Ath-96} that \[\chi^1_\A(q)=\left| \F_q^n-\bigcup_{H\in\A} H \right| .\] This result leads to the next problem.

\begin{prob}

If a matroid $M$ is realizable over a finite field of size $q$ then what does $\chi^k(q,\ldots ,q)$ count?
    
\end{prob}

\subsection{Free arrangements} Terao's factorization theorem (see \cite{Terao-80}) states that if an arrangement $\A$ is free (meaning the module of polynomial logarithmic derivations is a free module over the polynomial ring) then $\chi^1_\A(t)$ factors over the integers with roots given by the polynomial degrees of the generators.

\begin{prob}
Suppose an arrangement $\A$ is free. For $k>1$ is there some decomposition of $\chi^k_\A$ given by the generators of the module of polynomial logarithmic derivations? In particular, can one determine $\chi^k$ from the polynomial degrees of the generators or the betti numbers of the Jacobian ideal?
\end{prob}

Note that one can see in Example \ref{color-com} that these polynomials do not factor over the integers, yet they are well-known to be free arrangements. 

\subsection{Coefficients of $\chi^k_M$} Since the innovation by June Huh in \cite{Huh12} on chromatic polynomials of graphs there has been a renaissance in showing various properties, like log concavity (i.e. the coefficients $c-i$ satisfy $c_i^2\geq c_{i-1}c_{i+1}$), of the classic characteristic polynomials for arbitrary matroids (see \cite{AHK17}) and various other similar invariants (for example see \cite{BEST23,EH20,BH20}). Unfortunately, it turns out that, in general, not all coefficients of $\chi_M^k(t,\ldots ,t)$ form logarithmically concave sequences.

\begin{example}\label{notlog}
    Let $M_{K_4}$ be the matroid of the complete graph on 4 vertices. Then the 3rd chain characteristic polynomial condensed by evaluating at $t_1=t$, $t_2=t$, and $t_3=t$ is \[ \chi^3_{M_{K_4}}(t,t,t)= t^9 - 6t^8 + 17t^7 - 30t^6 + 37t^5 - 37t^4 + 30t^3 - 17t^2 + 11t - 6.\] Examining the second inequality to check we find that $17^2<30*11$. Hence, these coefficients do not form a logarithmically concave sequence.
\end{example}

While Example \ref{notlog} is discouraging, the sequence of coefficients is unimodal. Furthermore, we computed many examples of $\chi^2_M$ and found that in every example the sequence of coefficients was logarithmically concave. These computations led us to the following problem and conjecture.

\begin{prob}
 Do the coefficients of $\chi^k(t,\ldots ,t)$ for all $k$ or some specific $k$ satisfy any properties like unimodal? Is there some other evaluation of $\chi^k_M$ where the coefficients satisfy any properties?
\end{prob}

Given the relation to the M\"obius polynomial in Theorem \ref{Mobius} and many computations, we propose the following conjecture. 

\begin{conjecture}
    For any matroid, the coefficients of the polynomial $\chi_M^2(t,t)$ form a logarithmically concave sequence.
\end{conjecture}

\bibliographystyle{amsplain}

\bibliography{whitneyb}

\end{document}